\documentclass[12pt]{article}
\usepackage{amsfonts}
\usepackage{amssymb}
\usepackage{amsbsy}
\usepackage{amsthm}
\usepackage{amsmath}
\usepackage{amscd}
\usepackage{color}
\usepackage{cite}

\newtheorem{thm}{Theorem}[section]
\newtheorem{cor}[thm]{Corollary}
\newtheorem{lem}[thm]{Lemma}
\newtheorem{prop}[thm]{Proposition}
\theoremstyle{definition}
\newtheorem{defn}[thm]{Definition}
\theoremstyle{remark}
\newtheorem{rem}[thm]{Remark}
\newtheorem{ex}[thm]{\bf Example}

\newcommand{\bt}{\begin{thm}}
\newcommand{\et}{\end{thm}}
\newcommand{\bc}{\begin{cor}}
\newcommand{\ec}{\end{cor}}
\newcommand{\bl}{\begin{lem}}
\newcommand{\el}{\end{lem}}
\newcommand{\bp}{\begin{prop}}
\newcommand{\ep}{\end{prop}}
\newcommand{\bd}{\begin{defn}}
\newcommand{\ed}{\end{defn}}
\newcommand{\br}{\begin{rem}}
\newcommand{\er}{\end{rem}}
\newcommand{\bpr}{\begin{proof}}
\newcommand{\epr}{\end{proof}}
\newcommand{\bex}{\begin{ex}}
\newcommand{\eex}{\end{ex}}
\newcommand{\bcd}{\begin{CD}}
\newcommand{\ecd}{\end{CD}}

\newcommand{\bi}{\begin{itemize}}
\newcommand{\ei}{\end{itemize}}
\newcommand{\be}{\begin{enumerate}}
\newcommand{\ee}{\end{enumerate}}
\newcommand{\ba}{\begin{array}}
\newcommand{\ea}{\end{array}}
\newcommand{\beq}{\begin{equation}}
\newcommand{\eeq}{\end{equation}}
\newcommand{\beqa}{\begin{eqnarray}}
\newcommand{\eeqa}{\end{eqnarray}}
\newcommand{\bca}{\begin{cases}}
\newcommand{\eca}{\end{cases}}
\newcommand{\bal}{\begin{aligned}}
\newcommand{\eal}{\end{aligned}}

\newcommand{\ts}{\textstyle}

\newcommand{\N}{{\mathbb N}}
\newcommand{\Z}{{\mathbb Z}}
\newcommand{\R}{{\mathbb R}}
\newcommand{\C}{{\mathbb C}}
\newcommand{\T}{{\mathbb T}}
\newcommand{\D}{{\mathbb D}}

\newcommand{\bs}{\boldsymbol}

\newcommand{\spn}{{\operatorname{span}}}

\newcommand{\re}{{\operatorname{Re}}}
\newcommand{\im}{{\operatorname{Im}}}

\newcommand{\ad}{{\operatorname{ad}}}

\newcommand{\<}{\langle}
\renewcommand{\>}{\rangle}
\newcommand{\uk}{|\kern-3.3pt\uparrow\>}
\newcommand{\dk}{|\kern-3.3pt\downarrow\>}
\newcommand{\ub}{\<\uparrow\kern-3.3pt|}
\newcommand{\db}{\<\downarrow\kern-3.3pt|}

\begin{document}

\title{\bf The CMV bispectral problem}

\author{
F.A. Gr\"unbaum$^{1}$,
\; 
L. Vel\'azquez$^{2,}\footnote{Corresponding author: \texttt{velazque@unizar.es}}$
}

\date{\footnotesize
$^1$ Department of Mathematics, University of California, Berkeley, USA
\break 
$^2$ Departamento de Matem\'atica Aplicada \& Instituto Universitario de Matem\'aticas y Aplicaciones (IUMA), Universidad de Zaragoza, Spain
}

\maketitle

\begin{abstract}

A classical result due to Bochner classifies the orthogonal polynomials on the real line which are common eigenfunctions of a second order linear differential operator. We settle a natural version of the Bochner problem on the unit circle which answers a similar question concerning orthogonal Laurent polynomials and can be formulated as a bispectral problem involving CMV matrices. We solve this CMV bispectral problem in great generality proving that, except the Lebesgue measure, no other one on the unit circle yields a sequence of orthogonal Laurent polynomials which are eigenfunctions of a linear differential operator of arbitrary order. Actually, we prove that this is the case even if such an eigenfunction condition is imposed up to finitely many orthogonal Laurent polynomials. 
\end{abstract}

{\footnotesize

\smallskip

\noindent{\it (2010) AMS Mathematics Subject Classification}: 42C05, 47B36.

}

\section{Introduction} \label{sec:INTRO}

The motivation for the problem we address here can be traced to work in signal processing started by C.~Shannon \cite{S}. Addressing his problem required finding and exploiting some remarkable mathematical miracles and was accomplished in a series of papers by three workers at Bell Labs in the 1960's: David Slepian, Henry Landau and Henry Pollak, see \cite{SLP1,SLP2,SLP3,SLP4,S2,SLP5,S1}. The most important of these miracles is the existence of a second order {\it differential} operator that commutes with Shannon's time-and-band limiting {\it integral} operator.

In an effort to understand and extend the range of applicability of these miracles one of us introduced the so called ``bispectral problem", see \cite{DG}. For connections of this notion with the ``time-and-band limiting problem" of Shannon see for instance \cite{CG,G33,G1,G4,G6,GPZ}. The basic idea is that bispectral instances should lead to situations featuring the remarkable algebraic properties exploited by D.~Slepian, H.~Landau and H.~Pollak.
A strict connection between these two properties has not yet been
established.

These algebraic properties have important numerical/practical consequences. For a very recent account of several computational issues see \cite{BK14,JKS,ORX}. For new areas of applications involving (sometimes) vector-valued quantities on the sphere, see \cite{JB,PS,SD,SDW}. From a different numerical point of view see \cite{EMT}.

The study of the bispectral problem has moved in several fronts and led many
unsuspected areas of mathematics, for a sample see \cite{GT1,T2,GPT1,GPT5,GI,GHH,GH4,GH7,GH3,GDar,GY}.

While the initial problem of C.~Shannon was formulated in a continuous-continuos setup, the case of Fourier series (a discrete-continuous version) was handled by D.~Slepian in \cite{SLP5}, and the case of the DFT (a discrete-discrete version) was discussed in \cite{G9}. 

If one replaces the unit circle by the real line, these bispectral problems have a precedent in a continuous-discrete setup in the work of S.~Bochner \cite{B} and previous workers such as E.~Routh \cite{R}. They classified all families of orthogonal polynomials on the real line that admit a common second order differential operator having all of them as eigenfunctions. This constitutes a bispectral situation since orthogonal polynomials are (formal) eigenvectors of Jacobi matrices. The first step in going beyond second order differential operators was taken by \cite{Krall}. This issue was later addressed in other situations, such as second order q-difference equations, where the Askey-Wilson polynomials were found to be the most general case. If one gets away from polynomials the class of solutions is much larger, see \cite{GH4}. 

With all this as background we can state the contents of the paper: the natural extension of \cite{SLP5} by replacing the Lebesgue measure (the case of the Fourier series) by an arbitary measure on the unit circle leads to a new bispectral problem, which we consider here. This takes us back to \cite{Sz} who talked about orthogonal polynomials with respect to arbitrary measures on the unit circle. A better approach is taken in \cite{CMV,W} (see also \cite{Si,Si2}), where one applies the Gram-Schmidt process to all the integer powers --and not only the positive ones as in \cite{Sz}-- and obtains an orthonormal basis for the corresponding $L^2$ space. We study the bispectral problem for this basis of Laurent polynomials.

This bispectral problem constitutes the natural analogue on the unit circle of the Bochner problem on the real line. The role of the Jacobi matrices in the ad-conditions is played now by its unitary counterpart, the CMV matrices \cite{CMV,Si,Si2,W}, which encode the recurrence relation for the orthonormal basis of Laurent polynomials. In other words, our aim is to find all the orthonormal Laurent polynomials on the unit circle which are common eigenfunctions of a linear differential operator.
 
We will refer to this as the CMV bispectral problem since it can be formulated as a bispectral problem involving CMV matrices: to find all the CMV matrices whose (formal) eigenvectors, given by the corresponding orthonormal Laurent polynomials, are simultaneously eigenvectors of a linear differential operator.  

The ad-conditions introduced in \cite{DG} have been the main workhorse to study different bispectral situations \cite{GH3,GH4,GI}, and is the approach we are going to follow here, see Section~\ref{sec:BSP-CMV}. However, the standard ad-conditions become too messy to solve the CMV bispectral problem by applying them directly. Instead of this, we will exploit the unitarity and factorization properties of CMV matrices to transform the related ad-conditions into what we call the Hermitian ad-conditions because they come from the calculation of a Hermitian matrix. The result is a reduction of the ad-conditions in number and complexity, which allows us to solve them for second order linear differential operators, see Section~\ref{sec:herm-ad}.

Nevertheless, going beyond second order differential operators calls for more effective tools than solving ad-conditions by brute force. This is the aim of Sections~\ref{sec:ad-I} and \ref{sec:ad-F}, which develop the ad-integration and ad-factorization of ad-conditions. The ad-integration refers to a reduction in the order of the difference equations involved in the ad-conditions. The idea of solving ad-conditions by means of ad-integration was first advanced in \cite{GH3} and then fully developed in \cite{Haine} for the case of Jacobi matrices. The adaptation of this technique to CMV matrices is the objective of Section~\ref{sec:ad-I}. 

On the other hand, while the standard ad-conditions are defined in terms of the power of the ad-operator, given by a single commutator, the more useful Hermitian ad-conditions are not given by the power of any operator. Despite of this, Section~\ref{sec:ad-F} proves that the Hermitian ad-conditions factorize into lower order ones. 

These are the main tools to tackle the general CMV bispectral problem in Section~\ref{sec:GEN-BSP}. We not only solve the CMV bispectral problem for linear differential operators of arbitrary order, but also assuming the corresponding eigenfunction condition up to finitely many orthonormal Laurent polynomials. Furthermore, the solution to this problem follows from the answer to a more general `bispectral' question in which a tridiagonal matrix takes the place of the diagonal matrix of eigenvalues for the differential operator. In all these cases we find that the only solution to the CMV bispectral problem is given by the integer powers of a complex variable, which are the orthonormal Laurent polynomials related to the Lebesgue measure on the unit circle. 

This is in contrast with the very rich structure of the solutions to the analogous problem on the real line. As it is pointed out in the conclusions of Section~\ref{sec:CO}, this negative result should not be viewed as the end of the story, but could help us to focus our attention on those situations on the unit circle which could end in bispectral problems with non-trivial solutions. Besides, the triviality of the CMV bispectral problem can be used to test on the unit circle the not fully understood connections of bispectrality with the miracles behind the time-and-band limiting problem and its unexpected links with integrable systems.

\section{Bispectral CMV matrices and ad-conditions} \label{sec:BSP-CMV}

CMV matrices naturally arise in the study of orthogonality on the unit circle \cite{CMV,Si,Si2,W}. For each probability measure $\mu$ with an infinite support lying on the unit circle $\T:=\{z\in\C : |z|=1\}$ we can consider the sequences $(\chi_n)_{n\geq0}$ and $(x_n)_{n\geq0}$ of orthonormal Laurent polynomials (OLP) coming from the orthonormalization in $L^2_\mu$ of
$(1,z,z^{-1},z^2,z^{-2},\dots)$ and $(1,z^{-1},z,z^{-2},z^2,\dots)$ respectively. Both sequences are related by the substar operation in the vector space $\C[z,z^{-1}]$ of Laurent polynomials,
$$
 \chi_n(z) = x_{n*}(z),
 \kern40pt
 f_*(z) = \overline{f(1/\overline{z})} \quad \forall f\in\C[z,z^{-1}].
$$
The probability measures on $\T$ with infinite support are parametrized by the Verblunsky coefficients, a sequence $(\alpha_n)_{n\geq0}$ in the open unit disk $\D:=\{z\in\C : |z|<1\}$ which generates the OLP via the five term recurrence relations
\begin{align} 
 \label{eq:CX}
 & {\cal C}^t \bs\chi(z) = z\bs\chi(z), 
 \quad 
 \bs\chi = \begin{pmatrix} \chi_0 \\ \chi_1 \\ \vdots \end{pmatrix},
 \kern40pt
 {\cal C} \bs x(z) = z\bs x(z), 
 \quad 
 \bs x = \begin{pmatrix} x_0 \\ x_1 \\ \vdots \end{pmatrix}, &
 \\[3pt]
 \label{eq:CMV}
 & {\cal C} =
 \begin{pmatrix}
	\overline{\alpha}_0 &  \rho_0\overline{\alpha}_1 &  \rho_0\rho_1 
	&  0 &  0 &  0 &  0 &  \dots
	\\
	\rho_0 &  -\alpha_0\overline{\alpha}_1 &  -\alpha_0\rho_1 
	&  0 &  0 &  0 &  0 &  \dots
	\\
	0 &  \rho_1\overline{\alpha}_2 &  -\alpha_1\overline{\alpha}_2 
	&  \rho_2\overline{\alpha}_3 &  \rho_2\rho_3 &  0 &  0 &  \dots
	\\
	0 &  \rho_1\rho_2 &   -\alpha_1\rho_2 &  -\alpha_2\overline{\alpha}_3 
	&  -\alpha_2\rho_3  &  0 &  0 &  \dots
	\\
	0 &  0 &  0 &  \rho_3\overline{\alpha}_4 &  -\alpha_3\overline{\alpha}_4 
	&  \rho_4\overline{\alpha}_5 &  \rho_4\rho_5 &  \dots
	\\
	0 &  0 &  0 &  \rho_3\rho_4 &  -\alpha_3\rho_4 
	&  -\alpha_4\overline{\alpha}_5 &  -\alpha_4\rho_5 &  \dots
	\\
	\dots &  \dots &  \dots &  \dots &  \dots &  \dots &  \dots &  \dots
 \end{pmatrix}, &
\end{align}
where $\rho_n=\sqrt{1-|\alpha_n|^2}$. Both the five-diagonal unitary matrix $\cal C$ and its transpose ${\cal C}^t$ are named CMV matrices. The identities in \eqref{eq:CX} follow from the basic ones
\begin{equation} \label{eq:LM}
 \begin{aligned}
 	& z\bs x(z) = {\cal L} \bs\chi(z), 
 	& \quad & {\cal L} =
	\left( 
 	\begin{smallmatrix} 
 		\Theta_0 \\ & \Theta_2 \\ & & \Theta_4 \\[-4pt] & & & \ddots 
 	\end{smallmatrix}
	\right),
 	\\[3pt]
 	& \bs\chi(z) = {\cal M} \bs x(z),  
 	& & {\cal M} = 
	\left(
 	\begin{smallmatrix} 
 		1 \\ & \kern3pt \Theta_1 \\ & & \Theta_3 \\[-4pt] & & & \ddots 
 	\end{smallmatrix}
	\right),
 \end{aligned}
 \qquad 
 \Theta_n = 
 \begin{pmatrix} 
	\overline\alpha_n & \rho_n \\ \rho_n & -\alpha_n 
 \end{pmatrix},
\end{equation}
hence ${\cal C}={\cal L}{\cal M}$ and ${\cal C}^t={\cal M}{\cal L}$ factorize as a product of a couple of $2 \times 2$-block diagonal symmetric unitary matrices. Using the shift matrix  
\begin{equation} \label{eq:shift}
 S = 
 \left(
 	\begin{smallmatrix}
 	0 & \kern4pt 1
 	\\[3pt]
 	& \kern4pt 0 & \kern4pt 1
 	\\[3pt]
 	& & \kern4pt 0 & 1
 	\\[-4pt]
 	& & & \ddots & \ddots
 \end{smallmatrix}
 \right)
\end{equation}
and its adjoint $S^\dag$, these factors can be expressed as
\begin{equation} \label{eq:LM-shift}
\begin{aligned}
 & {\cal L} = {\cal A}_e + {\cal B}_e S + S^\dag {\cal B}_e, 
 & \qquad &
 {\cal M} = {\cal A}_o + {\cal B}_o S + S^\dag {\cal B}_o, &
 \\
 & {\cal A}_e = 
 \left(
 \begin{smallmatrix}
 	\overline\alpha_0 \\[2pt] & \kern-2pt -\alpha_0 
 	\\[2pt] & & \overline\alpha_2 \\[2pt] & & & \kern-2pt -\alpha_2
 	\\[-4pt]
 	& & & & \ddots
 \end{smallmatrix}
 \right), 
 \qquad
 & & {\cal A}_o = 
 \left(
 \begin{smallmatrix}
   	1 \\[2pt] & \kern5pt \overline\alpha_1 \\[2pt] & & \kern-2pt -\alpha_1 
 	\\[2pt] & & & \overline\alpha_2 \\[2pt] & & & & \kern-2pt -\alpha_2
 	\\[-4pt]
 	& & & & & \ddots
 \end{smallmatrix}
 \right),
 \\[2pt]
 & {\cal B}_e = 
 \left(
 \begin{smallmatrix}
 	\\[-2pt]
 	\rho_0 \\ & 0 \\ & & \rho_2 \\ & & & 0
 	\\[-5pt]
 	& & & & \ddots
 \end{smallmatrix}
 \right),
 \qquad
 & & {\cal B}_o = 
 \left(
 \begin{smallmatrix}
   	0 \\ & \rho_1 \\ & & 0 \\ & & & \rho_3 
 	\\[-5pt]
 	& & & & \ddots
 \end{smallmatrix}
 \right).
\end{aligned}
\end{equation}

For every $z\in\C\setminus\{0\}$, the relations in \eqref{eq:CX} identify $\bs x(z)$ and $\bs\chi(z)$ as formal eigenvectors with eigenvalue $z$ for the matrices ${\cal C}$ and ${\cal C}^t$ respectively. Actually, Proposition \ref{prop:ker} in the Appendix implies that $\bs x(z)$ and $\bs\chi(z)$ span the set of such formal eigenvectors. Therefore, the search for OLP which are also eigenfunctions of a linear differential operator can be understood as a bispectral problem, which we will call the CMV bispectral problem. Every linear differential operator 
\begin{equation} \label{eq:D}
 D = \sum_{k=0}^r D_k(z) \frac{d^k}{dz^k}
\end{equation} 
arising from the CMV bispectral problem maps $\C[z,z^{-1}]$ onto itself because any Laurent polynomial is a finite linear combination of OLP. Applying $D$ given by \eqref{eq:D} to the powers $z^k$ we see by induction on $k$ that the linear differential operators $D \colon \C[z,z^{-1}] \to \C[z,z^{-1}]$ are those with Laurent polynomial coefficients $D_k(z)$.

The two CMV bispectral problems related to the OLP $x_n$ or $\chi_n$ are essentially identical because both OLP are simultaneously eigenfunctions of a (different) linear differential operator. This is due to the equivalence
$$
 L x_n = \lambda_n x_n 
 \, \Leftrightarrow \, 
 L_* \chi_n = \overline\lambda_n \chi_n,   
 \qquad 
 \lambda_n\in\C,
$$
for any linear operator $L$ in $\C[z,z^{-1}]$, where $L_*$ is the linear operator in $\C[z,z^{-1}]$ defined by 
$$
 L_*f = (Lf_*)_* \quad \forall f\in\C[z,z^{-1}].
$$
In the case of a linear differential operator $D$, the substar operation yields also a linear differential operator $D_*$ of the same order as $D$, a fact that follows from the general composition law $(\widetilde{L}L)_*=\widetilde{L}_*L_*$ together with the substar of a single derivative, 
$$
 \left(\frac{d}{dz}\right)_* = -z^2 \frac{d}{dz}.
$$
Therefore, the relation
$$
 D x_n = \lambda_n x_n 
 \,\Leftrightarrow\, 
 D_* \chi_n = \overline\lambda_n \chi_n,
$$
shows that the CMV bispectral problem can be equivalently studied using $\cal C$ or ${\cal C}^t$. 

For concreteness, in what follows we will use the CMV matrix $\cal C$. In other words, we will search for OLP $x_n$ which are eigenfunctions of some linear differential operator $D$ of arbitrary order $r\geq1$, i.e.
\begin{equation} \label{eq:bsp}
 D x_n = \lambda_n x_n, \qquad \lambda_n\in\C,
\end{equation}
or equivalently
$$
 D \bs x = \Lambda \bs x, 
 \qquad 
 \Lambda =
 \left( 
 \begin{smallmatrix} 
 	\lambda_0 \\ & \lambda_1 \\ & & \lambda_2 \\[-4pt] & & & \ddots  
 \end{smallmatrix}
 \right).
$$
In this case the CMV matrix $\cal C$ related to $x_n$ will be called bispectral.

Actually, we will study a more general problem. 

First, we will ask for the linear differential operator $D$ in $\C[z,z^{-1}]$ to satisfy \eqref{eq:bsp} up to finitely many OLP. This is equivalent to stating that $D \bs x = \Omega \bs x$ with $\Omega$ diagonal up to a finite submatrix, i.e. $\Omega = \Omega_N \oplus \Lambda$ for some $N$, where $\Omega_N$ is the $N \times N$ principal submatrix of $\Omega$ and $\Lambda$ is diagonal.

Second, we will search for solutions of more general relations than \eqref{eq:bsp}, namely,
\begin{equation} \label{eq:tri}
 D x_n \in \spn \{x_{n-1},x_n,x_{n+1}\}.
\end{equation}
This three-term difference-differential equation reads as $D \bs x = \Omega \bs x$ with $\Omega$ tridiagonal. Indeed, we will assume \eqref{eq:tri} only from some index onwards, which is equivalent to stating that $\Omega$ is tridiagonal up to a finite submatrix.

These comments are the origin of the following definition. 

\begin{defn} \label{def:almost}
We say that an infinite matrix $\Omega$ is almost (tri)diagonal if it is (tri)diagonal up to a finite submatrix, i.e. $\Omega-\widetilde\Omega$ has finitely many non-zero entries for some infinite (tri)diagonal matrix $\widetilde\Omega$.
\end{defn}

With this terminology, the existence of a linear differential operator $D$ in $\C[z,z^{-1}]$ satisfying \eqref{eq:bsp} or \eqref{eq:tri} for large enough $n$ is equivalent to stating that $D \bs x = \Omega \bs x$ with $\Omega$ almost diagonal and almost tridiagonal respectively. 

All these cases are covered by the relation 
\begin{equation} \label{eq:bsp-band}
 D \bs x = \Omega \bs x, \qquad \Omega \text{ banded},
\end{equation} 
a general situation characterized by ``ad-conditions" (see \cite{DG}) involving the CMV matrix $\cal C$ related to $\bs x$. Such CMV ad-conditions, made explicit in Theorem~\ref{thm:ad-CMV} below, are given in terms of a linear operator $(\ad\,{\cal C})$ in the vector space of band matrices, defined by the commutator
$$
 (\ad\,{\cal C})\Omega = [{\cal C},\Omega] = {\cal C}\Omega-\Omega{\cal C}.
$$
By induction, its powers can be seen to have the explicit form
\begin{equation} \label{eq:ad power}
 (\ad\,{\cal C})^n \Omega = 
 \sum_{k=0}^n (-1)^k {n \choose k} {\cal C}^{n-k} \Omega {\cal C}^k. 
\end{equation} 
This operator is essential in the following result, key for this paper, which is a translation of the general ideas in \cite{DG} to the case of CMV matrices. It characterizes the relation \eqref{eq:bsp-band} in terms of CMV ad-conditions. Proposition \ref{prop:ker} in the Appendix will be crucial for the proof.  

In what follows, $I$ stands for the infinite identity matrix.

\begin{thm} \label{thm:ad-CMV}
Given a sequence $x_n$ of OLP on the unit circle with CMV matrix ${\cal C}$, the following conditions are equivalent for any band matrix $\Omega$:
\begin{itemize}
	\item[(i)] There is a linear differential operator $D$ of order at most $r$ 	such that $D \bs x = \Omega \bs x$.
	\item[(ii)] $(\ad\,{\cal C})^{r+1} \Omega =0$.  
\end{itemize}
\end{thm}

\begin{proof}
Condition {\it (i)} reads as $\Omega \bs x(z) \in \spn\{\bs x(z),\bs x'(z),\dots,\bs x^{(r)}(z)\}$, $z\in\C\setminus\{0\}$. In view of Proposition~\ref{prop:ker} in the Appendix, this is equivalent to $\Omega \bs x(z) \in \ker({\cal C}-zI)^{r+1}$, i.e. $({\cal C}-zI)^{r+1} \Omega \bs x(z) = 0$. Using the expansion 
\begin{equation} \label{eq:C-z power}
 ({\cal C}-zI)^n = \sum_{k=0}^n (-1)^k {n \choose k} z^k \, {\cal C}^{n-k},
\end{equation} 
together with \eqref{eq:CX} and \eqref{eq:ad power}, we find that
$$
\begin{aligned}
 ({\cal C}-zI)^{r+1} \Omega \bs x(z) = 0 
 & \; \Leftrightarrow \;
 \sum_{k=0}^{r+1} (-1)^k {r+1 \choose k} 
 {\cal C}^{r+1-k} \Omega \, {\cal C}^k \bs x(z) = 0
 \\
 & \; \Leftrightarrow \;
 (\ad\,{\cal C})^{r+1} \Omega \bs x(z) = 0.
\end{aligned}
$$
Due to the linear independence of the OLP, the last condition is equivalent to {\it (ii)}.
\end{proof}

The fact that the ad-conditions involve no information about the explicit form of the operator $D$, apart from its order, makes the previous characterization particularly useful for discovering bispectral situations. 

The linear differential operator $D$ involved in Theorem~\ref{thm:ad-CMV} has the freedom of a constant factor and an additive constant, which corresponds to the freedom of the band matrix $\Omega$ in a numerical factor and the addition of a multiple of the identity. Among the solutions $\Omega$ of the ad-conditions we must discard the multiples of the identity as trivial solutions corresponding to differential operators of order zero. In what follows, we will use the expression `linear differential operators' to refer only to those of order greater than zero, i.e. with the form \eqref{eq:D} and $D_r(z)\ne0$ for $r\ge1$.

The operator $(\ad\,{\cal C})$ has a symmetry, inherited from the unitarity of the CMV matrix $\cal C$, which will be further exploited in the next section.

\begin{prop} \label{prop:ad+}
For any CMV matrix $\cal C$ and any band matrix $\Omega$,
$$
 (\ad\,{\cal C})^n \Omega^\dag = 
 {\cal C}^n ((\ad\,{\cal C})^n\Omega)^\dag \, {\cal C}^n.
$$
\end{prop}

\begin{proof}
For $n=1$,
$$
 {\cal C}((\ad\,{\cal C})\Omega)^\dag{\cal C} =
 {\cal C}({\cal C}\Omega-\Omega{\cal C})^\dag{\cal C} =
 {\cal C}\Omega^\dag-\Omega^\dag{\cal C} =
 (\ad\,{\cal C})\Omega^\dag.
$$
Assuming the identity for an index $n$,
$$
\begin{aligned}
 {\cal C}^{n+1} ((\ad\,{\cal C})^{n+1}\Omega)^\dag \, {\cal C}^{n+1} 
 & = {\cal C} (\ad\,{\cal C})^n (((\ad\,{\cal C})\Omega)^\dag) \, {\cal C}
 \\
 & = (\ad\,{\cal C})^n ({\cal C}((\ad\,{\cal C})\Omega)^\dag{\cal C}) 
 = (\ad\,{\cal C})^{n+1}\Omega^\dag.
\end{aligned}
$$
\end{proof}

As a consequence of the previous result,
\begin{equation} \label{eq:H}
 (\ad\,{\cal C})^n\Omega = 0
 \;\Leftrightarrow\;
 (\ad\,{\cal C})^n\Omega^\dag = 0
 \;\Leftrightarrow\;
 \begin{cases}
  (\ad\,{\cal C})^n\re\,\Omega = 0, 
  & \re\,\Omega = \frac{1}{2}(\Omega+\Omega^\dag),
  \\
  (\ad\,{\cal C})^n\im\,\Omega = 0,
  & \ts \im\,\Omega = \frac{1}{2i}(\Omega-\Omega^\dag).
 \end{cases}
\end{equation}
This means that the solutions of the CMV ad-conditions can be chosen Hermitian without loss, since any solution $\Omega$ splits into Hermitian ones, $\re\,\Omega$ and $\im\,\Omega$. Therefore, in the CMV bispectral problem we can assume that the eigenvalues of the linear differential operator are real.

In view of Theorem~\ref{thm:ad-CMV} and the previous comments, the CMV bispectral problem can be reduced to the search for CMV matrices $\cal C$ with non-trivial real diagonal solutions $\Lambda$ of the CMV ad-conditions $(\ad\,{\cal C})^n\Lambda=0$ for some $n\ge2$.

\section{The Hermitian ad-conditions} \label{sec:herm-ad}

The CMV ad-conditions $(\ad\,{\cal C})^n\Lambda=0$ are too difficult to solve the problem directly in this way, but they can be rewriten in a more manageable form.

The matrix $(\ad\,{\cal C})^n\Lambda$ is $(4n+1)$-diagonal for any diagonal $\Lambda$, thus $(\ad\,{\cal C})^n\Lambda=0$ gives $4n+1$ difference equations, one for each diagonal. These difference equations are not all independent, so it should be possible to reorganize these ad-conditions in a smarter way. For this purpose we will introduce a narrower CMV ad-operator $(\ad_n\,{\cal C})$ which preserves hermiticity and such that $(\ad\,{\cal C})^n\Lambda=0$ iff $(\ad_n\,{\cal C})\Lambda=0$. This will reduce the number of difference equations. The key result is Proposition~\ref{prop:ad+} which shows that, when $\Lambda$ is real, $(\ad\,{\cal C})^n\Lambda={\cal C}^n((\ad\,{\cal C})^n\Lambda)^\dag\,{\cal C}^n$. Hence, we can get an Hermitian matrix by multiplying $(\ad\,{\cal C})^n\Lambda$ on the left and the right by ``half" of the matrix factors in $({\cal C}^\dag)^n$. This suggests the following definition.

\begin{defn} \label{def:H-ad} 
For any CMV matrix ${\cal C}={\cal L}{\cal M}$ and any band matrix $\Omega$ we define
$$
 (\ad_n\,{\cal C})\Omega :=
 \begin{cases}
	({\cal C}^\dag)^m 
	((\ad\,{\cal C})^n\Omega) 
	({\cal C}^\dag)^m, 
	& \quad n=2m,
	\\
	{\cal L}^\dag ({\cal C}^\dag)^m 
	((\ad\,{\cal C})^n\Omega) 
	({\cal C}^\dag)^m {\cal M}^\dag, 
	& \quad n=2m+1.
\end{cases}
$$
\end{defn}

Due to the unitarity of $\cal L$ and $\cal M$, 
\begin{equation} \label{eq:h-red}
 (\ad\,{\cal C})^n\Omega=0
 \; \Leftrightarrow \;
 (\ad_n\,{\cal C})\Omega=0.
\end{equation}
Therefore, Theorem~\ref{thm:ad-CMV} can be restated in the following way: given a sequence of OLP $x_n$ on the unit circle with CMV matrix $\cal C$, for any band matrix $\Omega$ the Hermitian ad-conditons $(\ad_{r+1}\,{\cal C})\Omega=0$ characterize the existence of a linear differential operator $D$ of order at most $r$ such that $D \bs x = \Omega \bs x$.  

Also, Proposition~\ref{prop:ad+} implies that 
\begin{equation} \label{eq:H-Had}
 (\ad_n\,{\cal C}) \Omega^\dag = ((\ad_n\,{\cal C}) \Omega)^\dag,
\end{equation}
so that $(\ad_n\,{\cal C})\Omega$ is Hermitian whenever $\Omega^\dag=\Omega$, a requirement that we can assume without loss. 

From the definition of $(\ad_n\,{\cal C})\Omega$ we obtain directly the recursion
\begin{equation} \label{eq:REC-Had}
 \begin{aligned}
  & (\ad_{n+1}\,{\cal C})\Omega =
  \begin{cases}
	{\cal M}((\ad_n\,{\cal C})\Omega){\cal M}^\dag
	- {\cal L}^\dag((\ad_n\,{\cal C})\Omega){\cal L}, 
	& \text{ even } n,
	\\
	{\cal L}((\ad_n\,{\cal C})\Omega){\cal L}^\dag
	- {\cal M}^\dag((\ad_n\,{\cal C})\Omega){\cal M}, 
	& \text{ odd } n,
  \end{cases}
  \\
  & (\ad_0\,{\cal C})\Omega = \Omega,
 \end{aligned}
\end{equation}
which allows us to find easily the explicit form of $(\ad_n\,{\cal C})\Omega$ for small values of $n$, 
{\footnotesize
$$
\kern-2pt
\begin{aligned}
 (\ad_1\,{\cal C})\Omega 
 & = {\cal M}\Omega{\cal M}^\dag - {\cal L}^\dag\Omega{\cal L}, 
 \\
 (\ad_2\,{\cal C})\Omega 
 & = {\cal L}{\cal M}\Omega{\cal M}^\dag{\cal L}^\dag 
 - 2\Omega
 + {\cal M}^\dag{\cal L}^\dag\Omega{\cal L}{\cal M},
 \\
 (\ad_3\,{\cal C})\Omega 
 & = {\cal M}{\cal L}{\cal M}\Omega{\cal M}^\dag{\cal L}^\dag{\cal M}^\dag
 - 3{\cal M}\Omega{\cal M}^\dag + 3{\cal L}^\dag\Omega{\cal L}
 - {\cal L}^\dag{\cal M}^\dag{\cal L}^\dag\Omega{\cal L}{\cal M}{\cal L},
 \\
 (\ad_4\,{\cal C})\Omega
 & = 
 {\cal L}{\cal M}{\cal L}{\cal M}
 \Omega
 {\cal M}^\dag{\cal L}^\dag{\cal M}^\dag{\cal L}^\dag
 - 4{\cal L}{\cal M}\Omega{\cal M}^\dag{\cal L}^\dag 
 + 6\Omega
 - 4{\cal M}^\dag{\cal L}^\dag\Omega{\cal L}{\cal M}
 + {\cal M}^\dag{\cal L}^\dag{\cal M}^\dag{\cal L}^\dag
 \Omega
 {\cal L}{\cal M}{\cal L}{\cal M}.
\end{aligned}
$$
}For an arbitrary value of $n$, using the expansion \eqref{eq:ad power} we obtain 
$$
\begin{aligned}
 (\ad_n\,{\cal C})\Omega = 
 & \, \big(\stackrel{r)}{\cdots}{\cal M}{\cal L}{\cal M}\big)
 \Omega
 \big({\cal M}^\dag{\cal L}^\dag{\cal M}^\dag\stackrel{r)}{\cdots}\big)
 + (-1)^r
 \big(\stackrel{r)}{\cdots}{\cal L}^\dag{\cal M}^\dag{\cal L}^\dag\big)
 \Omega
 \big({\cal L}{\cal M}{\cal L}\stackrel{r)}{\cdots}\big)
 \\
 & + \text{ narrower band matrices. }
\end{aligned}
$$

Therefore, when $\Lambda$ is real diagonal, bearing in mind that ${\cal M}\Lambda{\cal M}^\dag$ and ${\cal L}^\dag\Lambda{\cal L}$ are tridiagonal, we find that $(\ad_n\,{\cal C})\Lambda$ is a $(4n-1)$-diagonal Hermitian matrix, so the ad-conditions $(\ad_n\,{\cal C})\Lambda=0$ only lead to $2n$ difference equations, corresponding to the main and upper diagonals. We can write explicitly the equations of $(\ad_n\,{\cal C})\Lambda=0$ for the first values of $n$. We will order the equations running from the top upper diagonal ($2n^{\rm th}$ diagonal) to the main one ($1^{\rm st}$ diagonal), using the previous equations to simplify the new ones and omitting them when they yield no independent equation. Proceeding in this way we obtain the following results for $n=2,3$:

\bigskip

\renewcommand{\arraystretch}{1.25}

{\scriptsize 

\hspace{-19pt}
\begin{tabular}{|l|l|}
 \hline
 \multicolumn{2}{|c|}{$(\ad_2\,{\cal C})\Lambda=0$}
 \\ \hline
 $4^{\rm th}$ diagonal 
 & $(\lambda_{k+1}-\lambda_k)\alpha_k=0, \kern7pt k\geq1$
 \\ \hline
 $3^{\rm rd}$ diagonal 
 & $(\lambda_1-\lambda_0)\alpha_0=0$
 \\ \hline
 $2^{\rm nd}$ diagonal 
 & $(\lambda_2-\lambda_0)\alpha_0=0$
 \\
 & $(\lambda_{k+2}-\lambda_{k-1})\alpha_k=0, \kern7pt k\geq1$
 \\ \hline
 $1^{\rm st}$ diagonal 
 & $\lambda_2-2\lambda_0+\lambda_1=0$
 \\
 & $\lambda_3-2\lambda_1+\lambda_0=0$
 \\
 & $\lambda_{k+4}-2\lambda_{k+2}+\lambda_k=0, \kern7pt k\geq0$
\\ \hline
\end{tabular}
\hspace{3pt}
\begin{tabular}{|l|l|}
 \hline
 \multicolumn{2}{|c|}{$(\ad_3\,{\cal C})\Lambda=0$}
 \\ \hline
 $6^{\rm th}$ diagonal 
 & $(\lambda_{k+1}-\lambda_k)\alpha_k=0, \kern7pt k\geq2$
 \\ \hline
 $5^{\rm th}$ diagonal 
 & $(\lambda_2-\lambda_1)\alpha_1=0$
 \\ \hline
 $4^{\rm th}$ diagonal 
 & $(\lambda_3-\lambda_0)\alpha_1
 = (\lambda_1-\lambda_0)\alpha_0(\overline\alpha_0\alpha_1-\alpha_0)$
 \\
 & $(\lambda_{k+2}-\lambda_{k-1})\alpha_k=0, \kern7pt k\geq2$
 \\ \hline
 $3^{\rm rd}$ diagonal 
 & $(\lambda_2-\lambda_0)\alpha_0=0,$
 \kern7pt
 $(\lambda_1-\lambda_0)\alpha_0=0$
 \\ \hline
 $2^{\rm nd}$ diagonal 
 & $(\lambda_3-\lambda_0)\alpha_0=0,$
 \kern7pt
 $(\lambda_4-\lambda_0)\alpha_1=0$
 \\
 & $(\lambda_{k+3}-\lambda_{k-2})\alpha_k=0, \kern7pt k\geq2$
 \\ \hline
 $1^{\rm st}$ diagonal 
 & $\lambda_3-3\lambda_1+3\lambda_0-\lambda_2=0$
 \\
 & $\lambda_4-3\lambda_2+3\lambda_0-\lambda_1=0$
 \\
 & $\lambda_5-3\lambda_3+3\lambda_1-\lambda_0=0$
 \\
 & $\lambda_{k+6}-3\lambda_{k+4}+3\lambda_{k+2}-\lambda_k=0, \kern7pt k\geq0$
 \\ \hline
\end{tabular}

}

\bigskip

\noindent We can reorganize the above equations in a more natural way. In the following tables the equations with the same shape are grouped in the same column, except for the equations in red at the top of some columns, which are slightly different.

\bigskip

{\scriptsize

\bigskip

\begin{tabular}{|c|c|c|}
 \hline
 \multicolumn{3}{|c|}{$(\ad_2\,{\cal C})\Lambda=0$}
 \\ \hline
 Eq$_1$ & Eq$_2$ & RR
 \\ \hline
 $(\lambda_1-\lambda_0)\alpha_0=0$ 
 & {\color{red} $(\lambda_2-\lambda_0)\alpha_0=0$} 
 & {\color{red} $\lambda_2-2\lambda_0+\lambda_1=0$}
 \\ 
 $(\lambda_2-\lambda_1)\alpha_1=0$ 
 & $(\lambda_3-\lambda_0)\alpha_1=0$ 
 & {\color{red} $\lambda_3-2\lambda_1+\lambda_0=0$}
 \\
 $(\lambda_3-\lambda_2)\alpha_2=0$ 
 & $(\lambda_4-\lambda_1)\alpha_2=0$ 
 & $\lambda_4-2\lambda_2+\lambda_0=0$
 \\
 $(\lambda_4-\lambda_3)\alpha_3=0$ 
 & $(\lambda_5-\lambda_2)\alpha_3=0$ 
 & $\lambda_5-2\lambda_3+\lambda_1=0$
 \\[-3pt]
 $\vdots$ & $\vdots$ & $\vdots$
 \\ \hline
\end{tabular}

\bigskip

\begin{tabular}{|c|c|c|c|}
 \hline
 \multicolumn{4}{|c|}{$(\ad_3\,{\cal C})\Lambda=0$}
 \\ \hline
 Eq$_1$ & Eq$_2$ & Eq$_3$ & RR
 \\ \hline
 $(\lambda_1-\lambda_0)\alpha_0=0$ 
 & {\color{red} $(\lambda_2-\lambda_0)\alpha_0=0$} 
 & {\color{red} $(\lambda_3-\lambda_0)\alpha_0=0$} 
 & {\color{red} $\lambda_3-3\lambda_1+3\lambda_0-\lambda_2=0$}
 \\
 $(\lambda_2-\lambda_1)\alpha_1=0$ 
 & $(\lambda_3-\lambda_0)\alpha_1=0$ 
 & {\color{red} $(\lambda_4-\lambda_0)\alpha_1=0$} 
 & {\color{red} $\lambda_4-3\lambda_2+3\lambda_0-\lambda_1=0$}
 \\ 
 $(\lambda_3-\lambda_2)\alpha_2=0$ 
 & $(\lambda_4-\lambda_1)\alpha_2=0$ 
 & $(\lambda_5-\lambda_0)\alpha_2=0$ 
 & {\color{red} $\lambda_5-3\lambda_3+3\lambda_1-\lambda_0=0$}
 \\
 $(\lambda_4-\lambda_3)\alpha_3=0$ 
 & $(\lambda_5-\lambda_2)\alpha_3=0$ 
 & $(\lambda_6-\lambda_1)\alpha_3=0$ 
 & $\lambda_6-3\lambda_4+3\lambda_2-\lambda_0=0$
 \\
 $(\lambda_5-\lambda_4)\alpha_4=0$ 
 & $(\lambda_6-\lambda_3)\alpha_4=0$ 
 & $(\lambda_7-\lambda_2)\alpha_4=0$ 
 & $\lambda_7-3\lambda_5+3\lambda_3-\lambda_1=0$
 \\[-3pt]
 $\vdots$ & $\vdots$ & $\vdots$ & $\vdots$
 \\ \hline
\end{tabular}

\bigskip

}

Let us use the previous results to find, for instance, the CMV matrices $\cal C$ with non-trivial real diagonal solutions $\Lambda$ for $(\ad_2\,{\cal C})\Lambda=0$. All but the first two entries in the column RR of the corresponding table yield the recurrence relation 
$$
 \lambda_{k+4} - 2\lambda_{k+2} + \lambda_k = 0, 
 \qquad k\ge0,
$$
whose general solution is
$$
 \lambda_k = a_0 + a_1k + (-1)^k (b_0 + b_1k),
 \qquad a_i,b_i\in\R.
$$
Imposing the remaining two conditions of RR, 
$$
 \lambda_2-2\lambda_0+\lambda_1=0, \qquad \lambda_3-2\lambda_1+\lambda_0=0,
$$
yields $a_1=0$ and $b_1=2b_0$, i.e.
\begin{equation} \label{eq:RR-ad2}
 \lambda_k = a_0 + b_0 (-1)^k (1+2k).
\end{equation}
Then, if $\alpha_j\ne0$ for some index $j$, the equation $(\lambda_{j+1}-\lambda_j)\alpha_j=0$ of the column Eq$_1$ implies that 
$$
 0 = \lambda_{j+1}-\lambda_j = 4b_0 (-1)^{j+1} (1+j),
$$
so that $b_0=0$ and $\Lambda$ is a multiple of the identity. 

Therefore, the only non-trivial solutions may appear when $\alpha_k=0$ for all $k$. In this case the equations of the columns Eq$_1$ and Eq$_2$ are automatically satisfied and the general solution of $(\ad_2\,{\cal C})\Lambda=0$ is given by \eqref{eq:RR-ad2}, i.e. 
$$
 \lambda_k = \lambda_0 + (\lambda_0-\lambda_1) \frac{(-1)^k(1+2k)-1}{4},
 \qquad \lambda_0,\lambda_1\in\R.
$$
In other words,
\begin{equation} \label{eq:sol-Leb}
 \Lambda = 
 \left(
 \begin{smallmatrix}
 	0
 	\\[1pt]
 	& \kern-1pt -1
 	\\[1pt]
 	& & \kern4pt 1
 	\\[1pt]
 	& & & \kern-1pt -2
 	\\[1pt]
 	& & & & \kern4pt 2
 	\\[-4pt]
 	& & & & & \ddots
 \end{smallmatrix}
 \right)
\end{equation}
up to numerical factors and addition of multiples of the identity. This solution corresponds to the OLP
$$
 x_{2m-1}(z)=z^{-m}, \qquad x_{2m}(z)=z^m,
$$
associated with the Lebesgue measure on the unit circle, which satisfy $D \bs x = \Lambda \bs x$ for the first order linear differential operator
$$
 D = z\frac{d}{dz}.
$$

Using the results of the table for $(\ad_3\,{\cal C})\Lambda=0$ we find that a similar analysis works for these ad-conditions. Concerning the column RR, the first three equations impose on the general solution of the remaining equations 
$$
 \lambda_{k+6}-3\lambda_{k+4}+3\lambda_{k+2}-\lambda_k=0,
 \qquad k\ge0,
$$
given by
$$
 \lambda_k = a_0 + a_1k + a_2k^2 + (-1)^k (b_0 + b_1k + b_2k^2),
 \qquad a_i,b_i\in\R,
$$
the constraints $a_2=a_1$, $b_1=2b_0$ and $b_2=0$. If $\alpha_j\ne0$ for some $j$, using Eq$_1$ and Eq$_2$ we find again that $\Lambda$ is a multiple of the identity. This leaves as the only non-trivial solution that one related to the Lebesgue measure as in the previous case.
    
As a consequence we have the following version of Bochner theorem for OLP on the unit circle.

\begin{thm} \label{thm:ad3}
The only OLP on the unit circle which are eigenfunctions of a linear differential operator of order not greater than two are those orthonormal with respect to the Lebesgue measure.
\end{thm}

We have seen that the simplicity of the Hermitian ad-conditions is enough to deal with the CMV bispectral problem for linear differential operators of lower degree just by brute force. However, to go beyond this we need to further develop the machinery of the CMV ad-conditions.

\section{The CMV ad-conditions: ad-integration} \label{sec:ad-I}

The CMV ad-conditions involve commutators with a CMV matrix. Hence, the study of the centralizer of a CMV matrix can help us to find a short-cut to
the solution of such ad-conditions.

\begin{defn} \label{def:Z}
We denote by ${\cal Z}({\cal C})$ the centralizer of the CMV matrix ${\cal C}$ in the multiplicative group of infinite band matrices, i.e. 
$$
 {\cal Z}({\cal C}) := 
 \{\Omega \text{ band matrix} : [{\cal C},\Omega]=0\}.
$$
We can write ${\cal Z}({\cal C}) = \cup_{n\ge0} {\cal Z}_n({\cal C})$, where
$$
 {\cal Z}_n({\cal C}) := 
 \{\Omega \; (2n+1)\text{-diagonal matrix} : [{\cal C},\Omega]=0\}.
$$
\end{defn}

The centralizer of a CMV matrix among banded matrices can be explicitly determined. Concerning the result below, keep in mind that, due to the unitarity of a CMV matrix ${\cal C}$, its inverse ${\cal C}^{-1}={\cal C}^\dag$ is also banded.

\begin{prop} \label{prop:Z}
For any CMV matrix $\cal C$,
$$
\begin{aligned}
 & {\cal Z}({\cal C}) = \{f({\cal C}):f\in\C[z,z^{-1}]\},
 \\
 & {\cal Z}_{2m}({\cal C}) = {\cal Z}_{2m+1}({\cal C}) = 
 \{f({\cal C}):f\in\spn\{1,z^{-1},z,z^{-2},z^2,\dots,z^{-m},z^m\}\}.
\end{aligned}
$$
\end{prop}

\begin{proof}
If $x_n$ are the OLP related to ${\cal C}$ we know from Proposition~\ref{prop:ker} that, for every $z\in\C\setminus\{0\}$, $\bs x(z)$ spans the set of formal eigenvectors of ${\cal C}$ with eigenvalue $z$. Thus, due to the linear independence of the OLP, given a band matrix $\Omega$,
$$
 [{\cal C},\Omega]=0 
 \;\Leftrightarrow\; 
 [{\cal C},\Omega]\bs x(z)=0 
 \;\Leftrightarrow\; 
 ({\cal C}-zI)\Omega\bs x(z)=0 
 \;\Leftrightarrow\; 
 \Omega\bs x(z)=f(z)\bs x(z),
$$
for some function $f\colon\C\to\C$. From the last equality, 
$$
 f(z) = f(z)x_0(z) = \sum_k\Omega_{0,k}\,x_k(z),
$$
hence $f\in\C[z,z^{-1}]$ because $\Omega$ is banded. 

Also, if $\Omega$ is $(2n+1)$-diagonal,
$$
 f \in \spn\{x_k\}_{k=0}^n =
 \begin{cases}
 \spn\{1,z^{-1},z,z^{-2},z^2,\dots,z^{-m},z^m\}, & n=2m,
 \\
 \spn\{1,z^{-1},z,z^{-2},z^2,\dots,z^{-m},z^m,z^{-m-1}\}, & n=2m+1,
 \end{cases}
$$ 
so that, for some $a_j\in\C$, we have that $\Omega=\sum_{j=-m}^ma_j{\cal C}^j$ if $\Omega$ is $(4m+1)$-diagonal, while $\Omega=\sum_{j=-m-1}^ma_j{\cal C}^j$ when $\Omega$ is $(4m+3)$-diagonal. However, since ${\cal C}^j$ and ${\cal C}^{-j}=({\cal C}^\dag)^j$ are both strictly $(4j+1)$-diagonal, $a_{-m-1}=0$ in the last case because otherwise $\Omega$ would be $(2n+3)$-diagonal but not $(2n+1)$-diagonal.

This proves that
$$
\begin{aligned}
 & {\cal Z}({\cal C}) \subset \{f({\cal C}):f\in\C[z,z^{-1}]\},
 \\ 
 & {\cal Z}_{2m}({\cal C}),{\cal Z}_{2m+1}({\cal C}) \subset 
 \{f({\cal C}):f\in\spn\{1,z^{-1},z,z^{-2},z^2,\dots,z^{-m},z^m\}\}.
\end{aligned}
$$
The reverse inclusions are obvious.
\end{proof}

The above result permits the integration of the CMV ad-conditions: if $\cal C$ is a CMV matrix, for any band matrix $\Omega$, 
$$
 (\ad\,{\cal C})^{n+1}\Omega = 0 
 \;\Leftrightarrow\;
 (\ad\,{\cal C})^n\Omega = f({\cal C}), \quad f\in\C[z,z^{-1}].
$$
In the case of $\Omega$ diagonal we can say much more. Indeed, we will state the result for $\Omega$ tridiagonal because it needs no more effort due to the equality ${\cal Z}_{2m}({\cal C}) = {\cal Z}_{2m+1}({\cal C})$.

\begin{prop} \label{prop:INT}
If $\cal C$ is a CMV matrix and $\Omega$ is a tridiagonal matrix, then
$$
 (\ad\,{\cal C})^{n+1}\Omega = 0 
 \;\Leftrightarrow\; 
 (\ad\,{\cal C})^n\Omega=a{\cal C}^n 
 \;\Leftrightarrow\; 
 (\ad_n\,{\cal C})\Omega=aI, 
 \qquad a\in\C.
$$
\end{prop}

\begin{proof}
Let $\Omega$ be tridiagonal. Then, $(\ad\,{\cal C})^n\Omega$ is $(4n+3)$-diagonal and
$$
 (\ad\,{\cal C})^{n+1}\Omega = 0 
 \;\Leftrightarrow\; 
 (\ad\,{\cal C})^n\Omega \in {\cal Z}_{2n+1}({\cal C}) 
 \;\Leftrightarrow\; 
 (\ad\,{\cal C})^n\Omega = \sum_{j=-n}^n a_j{\cal C}^j, \; a_j\in\C.
$$

If $\Omega$ is Hermitian, Proposition~\ref{prop:ad+} states that $(\ad\,{\cal C})^n\Omega={\cal C}^n((\ad\,{\cal C})^n\Omega)^\dag{\cal C}^n$, hence
$$
\sum_{j=-r}^ra_j{\cal C}^j = {\cal C}^r \left(\sum_{j=-r}^r\overline a_j{\cal C}^{-j}\right) {\cal C}^r
= \sum_{j=-r}^r\overline a_j{\cal C}^{2r-j} = \sum_{j=r}^{3r}\overline a_{2r-j}{\cal C}^j.
$$
On the other hand, the set $\{{\cal C}^j\}_{j\in\Z}$ is linearly independent: if $\sum_jb_j{\cal C}^j=0$, $b_j\in\C$, then $0=\sum_jb_j{\cal C}^j\bs
x(z)=(\sum_jb_jz^j)\bs x(z)$, thus $\sum_jb_jz^j=0$ and $b_j=0$ for all $j$.
From these results we conclude that $a_j=0$ for $j\neq r$ and $a_r\in\R$, which proves the proposition for a Hermitian $\Omega$.

The result for a non-Hermitian $\Omega$ follows from \eqref{eq:H}. 
\end{proof}

Similar ad-integration techniques to those in Proposition~\ref{prop:INT} have been considered previously in \cite{Haine} for Jacobi bispectral problems. 
 
As an illustration of the ad-integration techniques, we will use them to present a simplified resolution of the ad-conditions $(\ad\,{\cal C})^2\Lambda=0$ for a diagonal matrix $\Lambda$. According to Proposition~\ref{prop:INT}, this is equivalent to solve $(\ad_1\,{\cal C})\Lambda\propto I$. Using the notation $\lambda_k$, $k\ge0$, for the diagonal coefficients of $\Lambda$, a simple calculation yields
$$
\begin{gathered}
 (\ad_1\,{\cal C})\Lambda = 
 \left(
 \begin{smallmatrix}
 	b_0 & -a_0
	\\[3pt]
	-\overline{a}_0 & -b_1 & \overline{a}_1
	\\[3pt]
	& a_1 & b_2 & -a_2
	\\[3pt]
	& & -\overline{a}_2 & -b_3 & \overline{a}_3
	\\[3pt]
	& & & a_3 & b_4 & -a_4
	\\[-2pt]
	& & & & \ddots & \ddots & \ddots
 \end{smallmatrix}
 \right),
 \\
 a_k = (\lambda_k-\lambda_{k+1})\rho_k\alpha_k,
 \qquad
 b_k = 
 \begin{cases}
 	(\lambda_0-\lambda_1)\rho_0^2, 
	& k=0,
 	\\
	(\lambda_{k-1}-\lambda_k)\rho_{k-1}^2 + 
	(\lambda_k-\lambda_{k+1})\rho_k^2,
	& k\ge1.
 \end{cases}
\end{gathered}
$$
Therefore,
$$
 (\ad_1\,{\cal C})\Lambda \propto I 
 \;\Leftrightarrow\;
 \begin{cases} 
 a_k=0,
 \\
 b_{k+1}=-b_k,
 \end{cases}
 \kern-9pt\Leftrightarrow\;
 \begin{cases}
 	(\lambda_k-\lambda_{k+1})\alpha_k=0,
	\\
	(\lambda_k-\lambda_{k+1})\rho_k^2 
	= (-1)^k (k+1) (\lambda_0-\lambda_1)\rho_0^2.
 \end{cases}
$$
If $\alpha_j\ne0$ for some $j$, the above relations imply that $\lambda_k=\lambda_{k+1}$ for all $k$, i.e. $\Lambda\propto I$. 
On the other hand, when $\alpha_k=0$ for all $k$ the condition $a_k=0$ is automatically satisfied, while $b_{k+1}=-b_k$ determines $\Lambda$ as in \eqref{eq:sol-Leb} up to numerical factors and addition of multiples of the identity.

\section{The Hermitian ad-conditions: ad-factorization} \label{sec:ad-F}

Another useful tool to deal with the CMV bispectral problem is the ad-factorization of the Hermitian ad-operator. The original ad-operator $(\ad\,{\cal C})^n$ is by definition a power of the simple ad-operator $(\ad\,{\cal C})$, but this is no longer true for the Hermitian ad-operator $(\ad_n\,{\cal C})$.  To understand the ad-factorization of $(\ad_n\,{\cal C})$ let us exploit again the possibility of approaching the CMV bispectral problem using two kinds of OLP, $x_n$ and $\chi_n$. 

Section \ref{sec:BSP-CMV} shows that the previous results about the bispectral problem for $x_n$ can be translated to the bispectral problem for $\chi_n$ just by performing the following transformations:
$$
\begin{matrix}
 & x_n & \longrightarrow & \chi_n
 \\
 & {\cal C} & \longrightarrow & {\cal C}^t
 \\
 & {\cal L} & \longrightarrow & {\cal M}
 \\
 & {\cal M} & \longrightarrow & {\cal L}
\end{matrix}
$$
For instance, the bispectral problem $D\bs\chi=\Lambda\bs\chi$ can be solved by using the ad-conditions $(\ad\,{\cal C}^t)^n\Lambda=0$. These ad-conditions are
equivalent to the Hermitian ones $(\ad_n\,{\cal C}^t)\Lambda=0$, where the definition and properties of $(\ad_n\,{\cal C}^t)$ can be obtained from those of $(\ad_n\,{\cal C})$ by simply making the exchanges ${\cal C} \leftrightarrow {\cal C}^t$ and ${\cal L} \leftrightarrow {\cal M}$.

A number of properties relate the ad-operators $(\ad_n\,{\cal C})$ and $(\ad_n\,{\cal C}^t)$, among them the ad-factorization that we are interested in. The following proposition summarizes these properties.

\begin{prop} \label{prop:C-Ct} 
Given a CMV matrix $\cal C$, the following relations hold for any band matrix $\Omega$:

\begin{itemize}

\item[(i)] $((\ad_n\,{\cal C})\Omega)^t = (-1)^n (\ad_n\,{\cal C}^t)\Omega^t$.

\item[(ii)] $(\ad_n\,{\cal C})({\cal L}\Omega{\cal M}) =
\begin{cases}
 {\cal L}((\ad_n\,{\cal C}^t)\Omega){\cal M}, & \text{ even } n,
 \\
 {\cal M}((\ad_n\,{\cal C}^t)\Omega){\cal L}, & \text{ odd } n.
\end{cases}$

\item[(iii)] $(\ad_n\,{\cal C})\Omega = 
(\ad_{n-k}\,{\cal C}(k))((\ad_k\,{\cal C})\Omega),
\quad 
{\cal C}(k) :=
\begin{cases}
 {\cal C}, & \text{ even } k,
 \\
 {\cal C}^t, & \text{ odd } k.
\end{cases}$

\end{itemize}
\end{prop}

\begin{proof}
Property {\it (i)} follows from Definition~\ref{def:H-ad} of $(\ad_n\,{\cal C})\Omega$ and the corresponding one for $(\ad_n\,{\cal C}^t)\Omega$, together
with the relation $((\ad\,{\cal C})^n\Omega)^t=(-1)^n(\ad\,{\cal C}^t)^n\Omega^t$, obtained iterating $((\ad\,{\cal C})\Omega)^t=-(\ad\,{\cal C}^t)\Omega^t$.

Analogously, the iteration of $(\ad\,{\cal C})({\cal L}\Omega{\cal M})={\cal L}((\ad\,{\cal C}^t)\Omega){\cal M}$ gives rise to the identity $(\ad\,{\cal C})^n({\cal L}\Omega{\cal M})={\cal L}((\ad\,{\cal C}^t)^n\Omega){\cal M}$.
Introducing in this equality Definition~\ref{def:H-ad} and its counterpart for $(\ad_n\,{\cal C}^t)\Omega$, when $n=2m+1$ leads to
$$
\begin{aligned}
 (\ad_n\,{\cal C})({\cal L}\Omega{\cal M}) 
 & = {\cal L}^\dag ({\cal C}^\dag)^m {\cal L} ({\cal C}^t)^m {\cal M} 
 ((\ad_n\,{\cal C}^t)\Omega) 
 {\cal L} ({\cal C}^t)^m {\cal M} ({\cal C}^\dag)^m {\cal M}^\dag =
 \\
 & = {\cal M} ((\ad_n\,{\cal C}^t)\Omega) {\cal L}.
\end{aligned}
$$
Here we have used that ${\cal L}({\cal C}^t)^m={\cal C}^m{\cal L}$ and $({\cal C}^t)^m{\cal M}={\cal M}{\cal C}^m$ due to the factorizations ${\cal C}={\cal L}{\cal M}$ and ${\cal C}^t={\cal M}{\cal L}$. This proves Property {\it (ii)} for odd $n$. The proof for even $n$ is similar.

Property {\it (iii)} is a direct consequence of Property {\it (ii)}. There are 4 cases to discuss depending on the parity of $n$ and $k$. We will show the proof for one of the cases, the others having a very similar proof. Consider an even $n=2m$ and an odd $k=2j+1$. Then, $n-k=2(m-j-1)+1$
is odd and
$$
\begin{gathered}
 (\ad\,{\cal C})^n\Omega 
 = {\cal C}^m ((\ad_n\,{\cal C})\Omega) {\cal C}^m,
 \qquad
 (\ad\,{\cal C})^k\Omega 
 = {\cal C}^j {\cal L} ((\ad_k\,{\cal C})\Omega) {\cal M} {\cal C}^j,
 \\
 (\ad\,{\cal C})^{n-k}\Omega 
 = {\cal C}^{m-j-1} {\cal L} 
 ((\ad_{n-k}\,{\cal C})\Omega) 
 {\cal M} {\cal C}^{m-j-1}.
\end{gathered}
$$
Using these relations and the factorization $(\ad\,{\cal C})^n\Omega = (\ad\,{\cal C})^{n-k} ((\ad\,{\cal C})^k\Omega)$ we get
$$
\begin{aligned}
 (\ad_n\,{\cal C})\Omega 
 & = {\cal C}^{-m} ((\ad\,{\cal C})^{n-k} ({\cal C}^j {\cal L} 
 ((\ad_k\,{\cal C})\Omega) 
 {\cal M} {\cal C}^j)) {\cal C}^{-m} =
 \\
 & = {\cal C}^{j-m} 
 ((\ad\,{\cal C})^{n-k} ({\cal L} ((\ad_k\,{\cal C})\Omega) {\cal M})) 
 {\cal C}^{j-m} =
 \\
 & = {\cal C}^{-1} {\cal L} ((\ad_{n-k}\,{\cal C})({\cal L} 
 ((\ad_k\,{\cal C})\Omega) {\cal M})) {\cal M} {\cal C}^{-1}.
\end{aligned}
$$
Finally, Property {\it (ii)} gives
$$
 (\ad_n\,{\cal C})\Omega 
 = {\cal C}^{-1} {\cal L} {\cal M} 
 ((\ad_{n-k}\,{\cal C}^t)((\ad_k\,{\cal C})\Omega)) 
 {\cal L}{\cal M} {\cal C}^{-1} 
 = (\ad_{n-k}\,{\cal C}^t) ((\ad_k\,{\cal C})\Omega).
$$
\end{proof}

Proposition~\ref{prop:C-Ct}.{\it (iii)} is the ad-factorization of $(\ad_n\,{\cal C})$. A special case is the recursive algorithm (\ref{eq:REC-Had}) for
$(\ad_n\,{\cal C})\Omega$ since it can be written as
$$
 (\ad_{n+1}\,{\cal C})\Omega 
 = (\ad_1\,{\cal C}(n)) ((\ad_n\,{\cal C})\Omega).
$$
The opposite special case, 
\begin{equation} \label{eq:ad-F-1}
 (\ad_{n+1}\,{\cal C})\Omega 
 = (\ad_n\,{\cal C}^t) ((\ad_1\,{\cal C})\Omega),
\end{equation}
will be particularly useful in dealing with the CMV bispectral problem for linear differential operators of arbitrary order.

\section{The general CMV bispectral problem} \label{sec:GEN-BSP}

We have proved that the CMV bispectral problem for linear differential operators of order not greater than two is trivial, i.e. the only bispectral CMV matrix is that one with null Verblunsky coefficients. The purpose of the present section is to generalize this result as much as possible by weakening the assumptions in different ways:
\begin{itemize}
\item[(A)] Admitting linear differential operators of arbitrary order.
\item[(B)] Requiring the eigenfunction condition with respect to the linear differential operator up to finitely many OLP. 
\item[(C)] Substituting the eigenfunction condition by the more general three-term difference-differential relation \eqref{eq:tri}.
\end{itemize} 

Remember that assuming (B) for a linear differential operator $D$ in $\C[z,z^{-1}]$ can be restated by saying that $D\bs x=\Omega\bs x$ with $\Omega$ almost diagonal (or almost tridiagonal if combined with (C)), i.e. $\Omega = \Omega_N \oplus \Lambda$ for some $N$, with $\Omega_N$ the $N \times N$ principal submatrix of $\Omega$ and $\Lambda$ diagonal. 

A first step in the direction pointed out in (B) is given by the following proposition.

\begin{prop} \label{prop:ad-C}
If $\cal C$ is a CMV matrix, then 
$$
 \Omega \text{ almost diagonal, } \; (\ad_1\,{\cal C})\Omega \text{ diagonal }
 \;\Rightarrow\;
 \Omega \text{ diagonal. }
$$
\end{prop}

\begin{proof}
Obviously, $\Omega$ is (almost) diagonal iff $\re\,\Omega$ and $\im\,\Omega$ are simultaneously (almost) diagonal. Also, from \eqref{eq:H-Had}, $\re\,((\ad_1\,{\cal C})\Omega)=(\ad_1\,{\cal C})\re\,\Omega$ and $\im\,((\ad_1\,{\cal C})\Omega)=(\ad_1\,{\cal C})\im\,\Omega$. Therefore, by taking real and imaginary parts, it suffices to prove the proposition for an Hermitian $\Omega$. 

By induction on $N$, it is enough to see that $\Omega=\Omega_N\oplus\Lambda$ with $\Lambda$ diagonal implies that $\Omega_N=\Omega_{N-1}\oplus \lambda_{N-1}$ with $\lambda_{N-1}\in\R$. If we write
$$
 \Omega_N = 
 \begin{pmatrix}
 \Omega_{N-1} & u_{N-1}
 \\
 u_{N-1}^\dagger & \lambda_{N-1}
 \end{pmatrix},
 \qquad
 u_{N-1}\in\C^{N-1},
 \qquad
 \lambda_{N-1}\in\R,
$$
all that we must prove is that $u_{N-1}=0$ whenever $(\ad_1\,{\cal C})\Omega$ is diagonal. 

Suppose that $\Omega=\Omega_N\oplus\Lambda$ for an even $N$, the proof for odd $N$ follows similar arguments. Then, denoting in general by the subscript $N$ the $N \times N$ principal submatrix and by the superscript $(N)$ the submatrix obtained deleting the first $N$ rows and columns, we have the splitting
$$
\begin{aligned}
 & A:={\cal L}^\dagger\Omega{\cal L} 
 = A_N \oplus A^{(N)},
 & \qquad & A_N = {\cal L}_N^\dagger\Omega_N{\cal L}_N,
 \\
 & B:={\cal M}\Omega{\cal M}^\dagger 
 = B_{N+1} \oplus B^{(N+1)},
 & & B_{N+1} = {\cal M}_{N+1}\Omega_{N+1}{\cal M}_{N+1}^\dagger,
\end{aligned}
$$
because ${\cal L}_N$ and ${\cal M}_{N+1}$ are direct sums of complete blocks $\Theta_k$. Since $\Omega_{N+1}=\Omega_N\oplus\lambda_N$ with $\lambda_N\in\R$ and 
$$
 {\cal M}_{N+1} 
 = {\cal M}_{N-1} \oplus \Theta_{N-1} = 
 \begin{pmatrix}
 {\cal M}_N & \rho_{N-1}e_N
 \\
 \rho_{N-1}e_N^\dagger & -\alpha_{N-1}
 \end{pmatrix},
 \qquad 
 e_N = 
 \begin{pmatrix} 0 \\[-5pt] \vdots \\[-5pt] 0 \\[-5pt] 1 \end{pmatrix}
 \in \C^N,
$$
we find that
$$
 B_{N+1} = 
 \begin{pmatrix}
 	B_N 
	& v_N 
	\\
	v_N^\dagger & *
 \end{pmatrix},
 \qquad
 v_N = \rho_{N-1}({\cal M}_N\Omega_N-\lambda_N\overline\alpha_{N-1}I_N)e_N.
$$
Therefore,
$$
 (\ad_1\,{\cal C})\Omega 
 = B-A =
 \left(
 \begin{array}{c|c}
 	B_N-A_N 
	& v_N \kern15pt 0
	\\ \hline
	\begin{array}{c} 
		v_N^\dagger \\ 0 \end{array} 
		& B^{(N)}-A^{(N)}
 	\end{array}
 \right),
$$
where $0$ stands for the null matrix of the appropriate size. Hence, $(\ad_1\,{\cal C})\Omega$ diagonal implies $v_N=0$. Bearing in mind that ${\cal M}_N={\cal M}_{N-1}\oplus\overline\alpha_{N-1}$ and ${\cal M}_{N-1}$ is unitary we conclude that
$$
 v_N = 0 
 \;\Rightarrow\;
 \lambda_N\overline\alpha_{N-1}e_N = {\cal M}_N\Omega_Ne_N =
 \begin{pmatrix} 
 	{\cal M}_{N-1}u_{N-1} \\[-3pt] \lambda_{N-1}\overline\alpha_{N-1} 
 \end{pmatrix} 
 \;\Rightarrow\; 
 u_{N-1} = 0.
$$
\end{proof}

Combining the previous proposition and the results of the calculations at the end of Section~\ref{sec:ad-I} we get the following result.

\begin{cor} \label{cor:ad-C}
If $\cal C$ is a CMV matrix whose Verblunsky coefficients are not all null, then 
$$
 \Omega \text{ almost diagonal, } \; (\ad_1\,{\cal C})\Omega \propto I
 \;\Rightarrow\;
 \Omega \propto I.
$$
\end{cor}

We have seen that an almost diagonal matrix satisfying certain ad-conditions must be actually diagonal. A similar result states that, under some ad-conditions, an almost tridiagonal matrix becomes almost diagonal.  

\begin{prop} \label{prop:atd-C}
If $\cal C$ is a CMV matrix, for any $n\in\N$, 
$$
 \Omega \text{ almost tridiagonal, } \; (\ad_n\,{\cal C})\Omega=0
 \;\Rightarrow\;
 \Omega \text{ almost diagonal. }
$$
\end{prop}

\begin{proof}
Supposing without loss that $\Omega$ is Hermitian, the fact that $\Omega$ is almost tridiagonal means that $\Omega = \widehat\Omega + \widetilde\Omega$, where $\widehat\Omega$ has finitely many non-null coefficients and $\widetilde\Omega$ is Hermitian tridiagonal. We can express 
$$
 \widetilde\Omega = 
 \left(
 \begin{smallmatrix}
 	\widetilde\lambda_0 & \kern3pt \lambda_0
	\\[3pt] 
	\overline\lambda_0 & \kern3pt \widetilde\lambda_1 & \lambda_1
	\\[3pt] 
	& \kern3pt \overline\lambda_1 & \widetilde\lambda_2 & \lambda_3
	\\[-3pt] 
	& & \ddots & \ddots & \ddots
 \end{smallmatrix}
 \right)
 = \widetilde\Lambda + \Lambda S + S^\dagger \Lambda^\dagger,
 \qquad
 \lambda_k\in\C,
 \qquad 
 \widetilde\lambda_k\in\R, 
$$
in terms of the shift matrix $S$ given in \eqref{eq:shift} and the two diagonal matrices 
$$ 
 \Lambda = 
 \left(
 \begin{smallmatrix}
 	\lambda_0 \\ & \lambda_1 \\ & & \lambda_2 \\[-5pt] & & & \ddots
 \end{smallmatrix}
 \right),
 \kern40pt 
 \widetilde\Lambda = 
 \left(
 \begin{smallmatrix}
 	\widetilde\lambda_0 
	\\ 
	& \kern2pt \widetilde\lambda_1 
	\\ 
	& & \kern2pt \widetilde\lambda_2 
	\\[-5pt] 
	& & & \ddots
 \end{smallmatrix}
 \right).
$$

We must prove that $\Omega$ is almost diagonal, i.e. $\lambda_k=0$ for big enough $k$, whenever $(\ad_n\,{\cal C})\Omega=0$. These ad-conditions imply that $(\ad_n\,{\cal C})\widetilde\Omega=-(\ad_n\,{\cal C})\widehat\Omega$ has only finitely many non-null coefficients. The conclusions of the proposition will follow from the analysis of the top upper diagonal of $(\ad_n\,{\cal C})\widetilde\Omega$, whose coefficients must vanish up to finitely many ones. 

To obtain the top upper diagonal in question it is useful to rewrite also $\cal L$ and $\cal M$ using the shift matrix, as in \eqref{eq:LM-shift}. The top upper diagonal of $(\ad_n\,{\cal C})\widetilde\Omega$ is the term corresponding to the highest power of the shift. In the case of even $n=2m$ such a term comes exclusively from the summands ${\cal C}^m\widetilde\Omega({\cal C}^\dagger)^m+({\cal C}^\dagger)^m\widetilde\Omega{\cal C}^m$ and is given by 
\begin{equation} \label{eq:lud-C}
 ({\cal B}_eS{\cal B}_oS)^m \Lambda S ({\cal B}_oS{\cal B}_eS)^m +
 ({\cal B}_oS{\cal B}_eS)^m \Lambda S ({\cal B}_eS{\cal B}_oS)^m. 
\end{equation}
Using the identity \eqref{eq:per-shift} to permute any diagonal matrix with the powers of the shift, \eqref{eq:lud-C} reads as $\Delta(n) S^{2n}$ where
$$
\begin{aligned}
 \Delta(n) =
 & \; {\cal B}_e{\cal B}_o^{(1)}{\cal B}_e^{(2)}{\cal B}_o^{(3)} 
 \cdots {\cal B}_e^{(n-2)}{\cal B}_o^{(n-1)}
 \Lambda^{(n)}
 {\cal B}_o^{(n+1)}{\cal B}_e^{(n+2)}{\cal B}_o^{(n+3)}{\cal B}_e^{(n+4)} 
 \cdots {\cal B}_o^{(2n-1)}{\cal B}_e^{(2n)} 
 \\ 
 & + {\cal B}_o{\cal B}_e^{(1)}{\cal B}_o^{(2)}{\cal B}_e^{(3)} 
 \cdots {\cal B}_o^{(n-2)}{\cal B}_e^{(n-1)}
 \Lambda^{(n)}
 {\cal B}_e^{(n+1)}{\cal B}_o^{(n+2)}{\cal B}_e^{(n+3)}{\cal B}_o^{(n+4)} 
 \cdots {\cal B}_e^{(2n-1)}{\cal B}_o^{(2n)}
 \\
 =
 & \left(
 \begin{smallmatrix}
	\\ \delta_0^{(n)} \\ & \delta_1^{(n)} \\[-2pt] & & & \ddots
 \end{smallmatrix}
 \right),
 \qquad
 \delta_k^{(n)} = 
 \rho_k\rho_{k+1}\cdots\rho_{k+n-1}
 \lambda_{k+n}
 \rho_{k+n+1}\rho_{k+n+2}\cdots\rho_{k+2n}. 
\end{aligned}
$$
Since $\Delta(n)$ must have finitely many non-null coefficients, we conclude that $\lambda_k=0$ for big enough $k$ in the case of even $n$. A similar proof works for odd $n$. 
\end{proof}

Propositions~\ref{prop:ad-C} and \ref{prop:atd-C}, as well as Corollary~\ref{cor:ad-C}, remain true when substituting ${\cal C}$ by ${\cal C}^t$, whose effect is simply exchanging ${\cal L}\leftrightarrow{\cal M}$. This can be used to obtain our main result.

\begin{thm} \label{thm:gen-C}
If $\cal C$ is a CMV matrix whose Verblunsky coefficients are not all null, for any $n\in\N$, 
$$
 \Omega \text{ almost tridiagonal, } \; (\ad_n\,{\cal C})\Omega=0
 \;\Rightarrow\;
 \Omega \propto I.
$$
\end{thm}

\begin{proof}
Suppose $\Omega$ almost tridiagonal satisfying $(\ad_n\,{\cal C})\Omega=0$. From Proposition~\ref{prop:atd-C} we know that $\Omega$ must be almost diagonal. Hence, $\Omega(1):=(\ad_1\,{\cal C})\Omega$ is almost tridiagonal and, according to the ad-factorization property \eqref{eq:ad-F-1}, $(\ad_{n-1}\,{\cal C}^t)\Omega(1)=(\ad_n\,{\cal C})\Omega=0$. Then, the result analogous to Proposition~\ref{prop:atd-C} for ${\cal C}^t$ implies that $\Omega(1)$ is almost diagonal. Hence, $\Omega(2):=(\ad_1\,{\cal C}^t)\Omega(1)$ is almost tridiagonal and satisfies $(\ad_{n-2}\,{\cal C})\Omega(2)=(\ad_{n-1}\,{\cal C}^t)\Omega(1)=0$ due to the analogue of the ad-factorization \eqref{eq:ad-F-1} for ${\cal C}^t$. Thus, Proposition~\ref{prop:atd-C} implies that $\Omega(2)$ is almost diagonal. Proceeding by induction we finally find almost diagonal matrices $\Omega(0)=\Omega,\Omega(1),\Omega(2),\dots,\Omega(n-1)$ such that 
$$
 (\ad_{n-k}\,{\cal C}(k))\Omega(k)=0,
 \qquad 
 \Omega(k+1) = (\ad_1\,{\cal C}(k))\,\Omega(k), 
 \qquad
 {\cal C}(k) = 
 \begin{cases}
 	{\cal C}, & \text{ even } k,
	\\
	{\cal C}^t,&\text{ odd } k.
 \end{cases}
$$
In particular, $\Omega(n-1)$ is almost diagonal and $(\ad_1\,{\cal C}(n-1))\Omega(n-1)=0$. From Corollary~\ref{cor:ad-C} we find that $\Omega(n-1)\propto I$. Hence, $(\ad_1\,{\cal C}(n-2))\,\Omega(n-2)\propto I$, so that $\Omega(n-2)\propto I$, again by Corollary~\ref{cor:ad-C}. Proceeding in this way we obtain by induction that $\Omega=\Omega(0)\propto I$. 
\end{proof}

The hypothesis of the previous theorem are equivalent to the existence of a linear differential operator $D$ such that $D \bs x = \Omega \bs x$ for an almost tridiagonal matrix $\Omega$, where $x_n$ are the OLP related to $\cal C$. This means that $D$ preserves $\C[z,z^{-1}]$ and $Dx_n\in\spn\{x_{n-1},x_n,x_{n+1}\}$ for all but finitely many indices $n$. Therefore, bearing in mind the equivalence \eqref{eq:h-red}, Theorem~\ref{thm:gen-C} has the following translation in terms of linear differential operators and OLP on the unit circle.

\begin{thm} \label{thm:bis-td-C}
The only OLP $x_n$ on the unit circle satisfying 
$$
 Dx_n\in\spn\{x_{n-1},x_n,x_{n+1}\}, 
 \qquad 
 \forall n \ge n_0,
 \qquad
 n_0\in\N,
$$ 
for a linear differential operator $D \colon \C[z,z^{-1}] \to \C[z,z^{-1}]$ of arbitrary order, are those orthonormal with respect to the Lebesgue measure.  
\end{thm}

As a particular case of this theorem we get the triviality of the general CMV bispectral problem. 

\begin{cor} \label{cor:bis-C}
The only OLP on the unit circle which, up to finitely many ones, are eigenfunctions of a linear differential operator $D \colon \C[z,z^{-1}] \to \C[z,z^{-1}]$ of arbitrary order, are those orthonormal with respect to the Lebesgue measure. 
\end{cor}

\section{Conclusions and outlook} \label{sec:CO}

We have shown that the CMV bispectral problem on the unit circle --at least in its traditional formulation, or even in some generalizations-- admits only the trivial solution. The question is: does this result close the topic? Our intention, based in other experiences involving the bispectral problem and its connections with the problem of C.~Shannon, is to keep looking in different directions in the context of the unit circle. Some hope is offered, for instance by results in \cite{HI,GI99,SpZh,Zh}, where one sees how getting away from polynomials leads to interesting situations. These references also show that considering not necessarily positive definite measures can be fruitful. The richness of the matrix valued Bochner problem in the case of the real line (whose full solution is still unknown, as a small sample see \cite{DG1,GPT1,GPT5,GT1}), compared to the scalar case, suggests that the triviality of the CMV bispectral problem may disappear if one admits matrix valued measures. A different path in this direction could arise from the use of more exotic kind of orthogonality on the unit circle, such as the one related to Sobolev inner products. All of this remains as a challenge.

It is important to point out that the Bochner-Krall problem --see \cite{Haine} for a very nice presentation-- is intimately connected with the study of the Toda lattice and its Virasoro symmetries. In connection with CMV matrices very relevant references are \cite{KN,Ne1,Ne2,HV}, where the Ablowitz-Ladik hierarchy is seen to be the integrable system that plays the role that the Toda lattice played for Jacobi matrices.

\section{Appendix} \label{app}

In this appendix we prove the following technical result, crucial for Theorem~\ref{thm:ad-CMV}.

\begin{prop} \label{prop:ker}
Let $x_n$ and $\chi_n$ be the OLP related to the CMV matrices $\cal C$ and ${\cal C}^t$ respectively. Then, a basis of $\ker({\cal C}-zI)^n$ is given by $\{\bs x(z),\bs x'(z),\dots,\bs x^{(n-1)}(z)\}$ and a basis of $\ker({\cal C}^t-zI)^n$ is given by $\{\bs\chi(z),\bs\chi'(z),\dots,\bs\chi^{(n-1)}(z)\}$ for every $z\in\C\setminus\{0\}$.
\end{prop}

\begin{proof}
We will prove the result for $\cal C$ and $x_n$, the proof for ${\cal C}^t$ and $\chi_n$ begin similar.

From \eqref{eq:CX} we find that $({\cal C}-zI) \bs x^{(k)}(z) = k \, \bs x^{(k-1)}(z)$ by induction on $k$. This leads to $({\cal C}-zI)^k \bs x^{(k)} = k! \, \bs x(z)$ and $({\cal C}-zI)^{k+1} \bs x^{(k)}(z) = 0$, which implies that $\spn\{\bs x(z),\bs x'(z),\dots,\bs x^{(n-1)}(z)\} \subset \ker({\cal C}-zI)^n$. Besides, $\{\bs x(z), \bs x'(z), \dots, \bs x^{(n-1)}(z)\}$ is linearly independent for every $z\ne0$ because applying $({\cal C}-zI)^{n-1}$ to the equation $c_0(z) \bs x(z) + c_1(z) \bs x'(z) + \cdots + c_{n-1}(z) \bs x^{(n-1)}(z) = 0$ yields $(n-1)! \, c_{n-1}(z) = 0$, so an induction gives $c_k(z)=0$ for all $k$. Hence, to prove that $\{\bs x(z),\bs x'(z),\dots,\bs x^{(n-1)}(z)\}$ is a basis of $\ker({\cal C}-zI)^n$ we only need to show that $\dim\ker({\cal C}-zI)^n=n$.

To determine $\dim\ker({\cal C}-zI)^n$ note that, due to the unitarity of $\cal L$ and $\cal M$, multiplying $({\cal C}-zI)^n$ on the left by ${\cal L}^\dag$ or ${\cal M}^\dag$ does not change its kernel. In particular, $\ker({\cal C}-zI)^n=\ker[K(n)({\cal C}-zI)^n]$, where
$$
 K(2m) = ({\cal C}^\dag)^m, 
 \kern40pt 
 K(2m+1) = {\cal L}^\dag ({\cal C}^\dag)^m.
$$
The advantage of $K(n)({\cal C}-zI)^n$ over $({\cal C}-zI)^n$ is its narrower band structure, which is shown by inserting the expansion \eqref{eq:C-z power}, so that 
$$
\begin{aligned}
& \kern-5pt 
 \begin{aligned}
   	K(2m) ({\cal C}-zI)^{2m}
 	& = \sum_{j=-m}^m (-1)^{m-j} {2m \choose m-j} z^{m-j} {\cal C}^j
 	\\
 	& = (-1)^{m} {2m \choose m} z^{m}
 	+ \sum_{j=1}^m (-1)^{m-j} {2m \choose m-j} 
 	\left(z^{m-j} {\cal C}^j + z^{m+j} ({\cal C}^\dag)^j \right),
 \end{aligned}
\\
& \kern-5pt 
 \begin{aligned}
 	K(2m+1) ({\cal C}-zI)^{2m+1}
 	& = \sum_{j=-m-1}^m (-1)^{m-j} {2m+1 \choose m-j} 
	z^{m-j} {\cal M} {\cal C}^j
	\\
	& = \sum_{j=0}^m (-1)^{m-j} {2m+1 \choose m-j} 
	\left(z^{m-j} {\cal M} {\cal C}^j 
	- z^{m+j+1} {\cal L}^\dag ({\cal C}^\dag)^j\right).
 \end{aligned}
\end{aligned}
$$
The above expressions prove that $K(n)({\cal C}-zI)^n$ is $(2n+1)$-diagonal. Besides, the coefficients of its top upper diagonal are non-null for $z\ne0$. To check this last statement note that this top upper diagonal comes exclusively from the terms 
$$
\begin{aligned}
 & {\cal C}^m + z^n ({\cal C}^\dag)^m, 
 & \quad & n=2m,
 \\
 & {\cal M} {\cal C}^m - z^n {\cal L}^\dag ({\cal C}^\dag)^m, 
 & & n=2m+1,
\end{aligned}
$$
and corresponds to the power $S^n$ when expanding in powers of the shift $S$ given in \eqref{eq:shift}. Using \eqref{eq:LM-shift} we find that such upper diagonal is 
$$
\begin{aligned}
 & ({\cal B}_e S {\cal B}_o S)^m 
 + z^n ({\cal B}_o S {\cal B}_e S)^m, 
 & \quad & n=2m,
 \\
 & {\cal B}_o S ({\cal B}_e S {\cal B}_o S)^m 
 - z^n {\cal B}_e S({\cal B}_o S {\cal B}_e S)^m, 
 & \quad & n=2m+1.
\end{aligned}
$$ 
We can permute any diagonal matrix $\Lambda$ with the powers of the shift $S$ via the identity 
\begin{equation} \label{eq:per-shift}
 S^k\Lambda=\Lambda^{(k)}S^k,
\end{equation}
where $\Lambda^{(k)}$ is obtained by deleting the first $k$ rows and columns of $\Lambda$. Therefore, the top upper diagonal in question reads as $\Gamma(n) S^n$ with
$$
 \Gamma(n) =
 \begin{cases}
 	{\cal B}_e {\cal B}_o^{(1)} 
	\cdots {\cal B}_e^{(n-2)} {\cal B}_o^{(n-1)}
 	+ z^n {\cal B}_o {\cal B}_e^{(1)} 
	\cdots {\cal B}_o^{(n-2)} {\cal B}_e^{(n-1)}, 
 	& \quad n=2m,
  	\\
 	{\cal B}_o {\cal B}_e^{(1)}  
	\cdots {\cal B}_e^{(n-2)} {\cal B}_o^{(n-1)}
 	- z^n {\cal B}_e {\cal B}_o^{(1)} 
	\cdots {\cal B}_o^{(n-2)} {\cal B}_e^{(n-1)}, 
 	& \quad n=2m+1.
 \end{cases}
$$
More explicitly,
$$
 \Gamma(n) = 
 \left(
 \begin{smallmatrix}
 	\\
 	\gamma_0^{(n)} \\ & \gamma_1^{(n)} \\[-2pt] & & \ddots
 \end{smallmatrix}
 \right),
 \quad
 \gamma_k^{(n)} = 
 \begin{cases}
 \rho_k\rho_{k+1}\cdots\rho_{k+n-1}, & k,n \text{ same parity, } 
 \\
 (-1)^nz^n\rho_k\rho_{k+1}\cdots\rho_{k+n-1}, & k, n \text{ different parity, }
 \end{cases}
$$
which clearly has non-null diagonal coefficients for $z\ne0$.

The fact that $K(n)({\cal C}-zI)^n$ is $(2n+1)$-diagonal with non-null coefficients in the top upper diagonal implies that $\dim\ker[K(n)({\cal C}-zI)^n]=n$. In other words, $\dim\ker({\cal C}-zI)^n=n$, which ends the proof of the proposition.
\end{proof}

\bigskip

\noindent{\Large \bf Acknowledgements}

The work of L. Vel\'azquez has been partially supported by the Spanish Government together with the European Regional Development Fund (ERDF) under grants MTM2011-28952-C02-01 (from Ministerio de Ciencia e Innovaci\'on of Spain) and MTM2014-53963-P (from Ministerio de Econom\'{\i}a y Competitividad of Spain), and by Project E-64 of Diputaci\'on General de Arag\'on (Spain). This author would like to thank also the Department of Mathematics from UC Berkeley for its hospitality during a stay where this work was partially developed.


\end{document}